\documentclass[english, 11pt]{amsart}
\usepackage{amssymb}
\usepackage{amsfonts}
\usepackage{amscd}
\usepackage[T1]{fontenc}
\usepackage{color}
\usepackage{tikz}

\usepackage{mathrsfs} 

\newcommand{\Vol}{\operatorname{Vol}}

\def\ac{{\overline{\rm ac}}}

\def\Def{\operatorname{Def}}

\def\LPas{\cL_{\rm DP}}

\def\charac{\operatorname{char}}

\def\11{{\mathbf 1}}

\def\FF{{\mathbb F}}

\def\LL{{\mathbb L}}

\def\NN{{\mathbb N}}

\def\QQ{{\mathbb Q}}
\def\RR{{\mathbb R}}

\def\ZZ{{\mathbb Z}}

\def\cA{{\mathcal A}}

\def\cC{{\mathcal C}}

\def\cL{{\mathcal L}}
\def\cM{{\mathcal M}}

\def\cO{{\mathcal O}}

\def\cQ{{\mathcal Q}}

\def\cT{{\mathcal T}}

\newtheorem{thm}[subsubsection]{Theorem}
\newtheorem{lem}[subsubsection]{Lemma}
\newtheorem{cor}[subsubsection]{Corollary}
\newtheorem{prop}[subsubsection]{Proposition}

\newtheorem{maintheorem}[subsubsection]{Main Theorem}

\theoremstyle{definition}
\newtheorem{defn}[subsubsection]{Definition}
\newtheorem{example}[subsubsection]{Example}

\newtheorem{def-prop}[subsubsection]{Proposition-Definition}
\newtheorem{def-theorem}[subsubsection]{Theorem-Definition}
\newtheorem{def-lem}[subsubsection]{Lemma-Definition}

\theoremstyle{remark}
\newtheorem{remark}[subsubsection]{Remark}

\theoremstyle{plain}

\numberwithin{equation}{subsection}

  {\par\medskip\noindent #1\par\begingroup%
    \advance\leftskip by 1em\advance\rightskip by 1em}%
  {\par\endgroup}

\newcommand{\ord}{\operatorname{ord}}

\def\VF{\mathrm{VF}}

\def\VG{\mathrm{VG}}
\newcommand{\RF}{{\rm RF}}

\begin{document}

\author{Kien Huu Nguyen}
\address{Universit\'e Lille 1, Laboratoire Painlev\'e, CNRS - UMR 8524, Cit\'e Scientifique, 59655 Villeneuve d'Ascq Cedex, France and  Department of Mathematics
Hanoi National University of Education 136 XuanThuy str., CAU GIAY, Hanoi, Vietnam}
\email{hkiensp@gmail.com}



\title[Definable, analytic equivalence relations on local fields]
{Uniform rationality of the Poincar\'e series of definable, analytic equivalence relations on local fields}

\begin{abstract} Poincar\'e series of $p$-adic, definable equivalence relations have been studied in various cases since Igusa's and Denef's work related to counting solutions of polynomial equations modulo $p^n$ for prime $p$. General semi-algebraic equivalence relations on local fields have been studied uniformly in $p$ recently in \cite{16}. Here we generalize the rationality result of \cite{16} to the analytic case, unifomly in $p$, building further on the appendix of \cite{16} and on \cite{13b}, \cite{03}. In particular, the results hold for large positive characteristic local fields. We also introduce rational motivic constructible functions and their motivic integrals, as a tool to prove our main results.
\end{abstract}

\maketitle


\section{Introduction}

\subsection{}

After observing that Igusa's and Denef's rationality results (see e.g. \cite{09}, \cite{09b}) can be rephrased in terms of counting the number of equivalence classes of particular semi-algebraic equivalence relations, it becomes natural to consider more general definable equivalence relations in the $p$-adic context and study the number of equivalence classes and related Poincar\'e series. 
The study of uniform $p$-adic, semi-algebraic equivalence relations is one of the main themes of \cite{16}, with general rationality results of the associated Poincar\'e series as part of the main results, generalizing \cite{22P} and \cite{MacUnif}. In the appendix of \cite{16}, a more direct way of obtaining such rationality results was developed in a different case, namely, in the subanalytic setting on $\QQ_p$, generalizing the rationality results by Denef and van den Dries in \cite{11}. A deep tool of \cite{16} to study equivalence relations, called elimination of imaginaries, is very powerful but also problematic since it does not extend well to the analytic setting, see \cite{15}.
In this paper we follow the more direct approach of the appendix of \cite{16} to obtain rationality results in situations where elimination of imaginaries is absent; here, we make this approach uniform in non-archimedean local fields (non-archimedean locally compact field of any characteristic). The two main such situations where this applies come from analytic structures on the one hand, and from an axiomatic approach from \cite{07} on the other hand. In the analytic case, our results generalize the uniform analytic results of \cite{13b}, \cite{03}. We heavily rely on
cell decomposition, a tool which was not yet available at the time of \cite{13b}, and which is obtained more recently in an analytic context in \cite{06}, \cite{06b} and \cite{03} uniformly, and in \cite{02} for any fixed $p$-adic field.
In our approach we also need more general denominators than in previous studies, which we treat by introducing rational motivic constructible functions and their motivic integrals, a slight generalization of the motivic constructible functions from \cite{07}. The adjective `rational' reflects the extra localization of certain Grothendieck semi-rings as compared to \cite{07}. For these integrals to be well-defined, a property called Jacobian property is used; also this property was not yet available at the time of \cite{13b}, and is shown in \cite{06} for analytic structures.

\subsection{}

Let us begin by rephrasing some of the classical results by e.g.~Igusa and Denef in terms of definable equivalence relations.
Let $\mathcal{L}_{0}$ be a first order, multi-sorted language such that $(\QQ_{p},\ZZ,\FF_p)$ are $\mathcal{L}_{0}$-structures for all $p$. A basic example is the ring language on the first sort together with the valuation map $\ord:\QQ_p^\times\to \ZZ$ and the ordering on $\ZZ$.
We consider a formula $\varphi(x,y,n)$ in the language $\mathcal{L}_{0}$ with free variables $x$ and $y$ running over $\QQ_{p}^{m}$ and $n\in\ZZ$. 

Suppose that for each $n\geq 0$ and each prime $p$, the condition $\varphi(x,y,n)$ on $x,y$ yields an equivalence relation $\sim_{p,n}$ on $\QQ_{p}^{m}$ (or on a uniformly definable subset $X_p$ of $\QQ_{p}^{m}$) with finitely many, say $a_{\varphi,p,n}$, equivalence classes.
Then we can consider, for each $p$, the Poincar\'e series associated to $\varphi$ and $p$: 
\begin{equation}\label{eq:P}
P_{\varphi,p}(T)=\sum_{n\geq 0}a_{\varphi,p,n}T^{n}
\end{equation}

For the case is that $\varphi(x,y,n)$ is the collection of equivalence relations $\sim_n$ based on the vanishing of a polynomial $f(x)$ modulo $p^n$, more precisely, for $x,y$ tuples with nonnegative valuation one requires
\begin{equation}\label{eq:f}
\ord ( f(x) , x-y ) \geq n
\end{equation}
with $f\in\ZZ[x_{1},...,x_{m}]$ and where the order of a tuple is the minimum of the orders, the question of the rationality of $P_{\varphi,p}(T)$ was conjectured by Borevich and Shafarevich in \cite{01} and was proved by Igusa in \cite{20} and \cite{21}. 
The proof relied on Hironaka's resolution of singularities and used $p$-adic integration. In \cite{09}, still using $p$-adic integration but using cell decomposition instead of Hironaka's resolution of singularities, Denef proved the rationality of $P_{\varphi,p}$ for more general $\varphi$ than in  (\ref{eq:f}) related to lifting solutions modulo $p^n$ to solutions in $\ZZ_p^m$, answering on the way a question given by Serre in \cite{23}. 
The idea of Denef (using \cite{17} and \cite{08Co}) is to represent a semi-algebraic set by a union of finitely many cells which have a simple description (and so does $\ord f$ on each cell, for $f$ a semi-algebraic function) so that we can integrate easily on them. This was made uniform in $p$ in \cite{22P} and \cite{MacUnif}. The advantage of the approach via cell decomposition is that also more general parameter situations can be understood via parameter integrals, a feature heavily used to get Fubini theorems for motivic integrals \cite{03}, \cite{04}, \cite{07}; for us this approach leads to our main result Theorem \ref{****} below, as a natural generalization of our rationality results Theorems \ref{**} and \ref{***}.

When $f$ in (\ref{eq:f}) is given by a converging power series instead of a polynomial, then the rationality has been obtained in \cite{11} for fixed $p$, in \cite{13b} uniformly in $p$, and in \cite{03} uniformly in $p$ with the extra strength coming from the cell decomposition approach.

The rationality of $P_{\varphi,p}(T)$ as in (\ref{eq:P}) for more general $\varphi$, as well as the uniformity for large $p$ in $\QQ_p$ and in $\FF_p((t))$, is the focus of this paper. A common feature of all the mentioned results is to bundle the information of $P_{\varphi,p}(T)$ into a $p$-adic integral of some kind, and then study these integrals qualitatively. Here, we bundle the information into slightly more general integrals than the ones occurring before.  

\subsection{}
Let us recall the precise result of \cite{16} which states the rationality of $P_{\varphi,p}(T)$ in the semi-algebraic realm, with its uniformity in (non-archimedean) local fields. We recall that a non-archimedean local field is a locally compact topological field with respect to a non-discrete topology such that its topology defines a non-archimedean absolute value. We have known that  any such field is either a finite field extension of $\QQ_p$ for some $p$ or isomorphic to $\FF_q((t))$ for some prime power $q$; we will from now on say \emph{local field} for non-archimedean local field.

Let $\LPas$ be the Denef-Pas language, namely with the ring language on the valued field sort and on the residue field sort, the language of ordered abelian groups on the value group, the valuation map, and an angular component map $\ac$ from the valued field sort to the residue field sort (see Section  \ref{sec:DP}). All local fields $K$ with a chosen uniformizer $\varpi$ are $\LPas$-structures, where the angular component map sends nonzero $x$ to the reduction of $x\varpi^{-\ord x}$ modulo the maximal ideal, and sends zero to zero. Let $\cO_K$ denote the valuation ring of $K$ with maximal ideal $\cM_K$ and residue field $k_K$ with $q_K$ elements and characteristic $p_K$.

Let $\varphi(x,y,n)$ be an $\LPas$-formula with free variables $x$ running over $K^m$, $y$ running over $K^m$ and $n$ running over $\NN$, with $\NN$ being the set nonnegative integers.
Suppose that for each local field $K$ and each $n$, $\varphi(x,y,n)$ gives an equivalence relation $\sim_{K,n}$ on $K^m$ with finitely many, say, $a_{\varphi,K,n}$, equivalence classes. (The situation that $\varphi(x,y,n)$ gives an equivalence relation on a uniformly definable subset $X_{K,n}$ of $K^m$ for each $n$ can be treated similarly, e.g.~by extending with a single extra equivalence class to extend to a relation on $K^m$.)
For each local field $K$ consider the associated Poincar\'e series
$$
P_{\varphi,K}(T) =\sum_{n\geq 0}a_{\varphi,K,n}T^{n},
$$

In \cite{16}, the authors proved the following (as well as a variant by adding constants of a ring of integers to $\LPas$, and by allowing $n$ and $T$ to be tuples of finite length instead of length one; these features are also captured in Theorem \ref{****} below).

\begin{thm}\label{*}
There exists $M>0$ such that the power series  $P_{\varphi,K}(T)$ is rational in $T$ for each local field $K$ whose residue field has characteristic at least $M$. Moreover, for such $K$, the series $P_{\varphi,K}(T)$ only depends on the residue field $k_K$ (namely, two local fields with isomorphic residue field give rise to the same Poincar\'e series).

More precisely, there exist nonnegative integers $a,N,M,k,b_j,e_i,q$, integers $a_j$, and formulas $X_{i}$ in the ring language for $i=0,\ldots,N$ and $j=0,\ldots,k$, such that for each $j, a_{j}$ and $b_{j}$ are not both $0$, $q$ is nonzero, and for all local fields $K$ with residue field of characteristic at least $M$ we have
$$
P_{\varphi,K}(T) = \frac{\sum\limits_{i=0}^{N} (-1)^{e_i}  \#X_{i}(k_K)T^{i}}{q \cdot  q_K^{a}\prod_{j=1}^{k}(1-q_K^{a_{j}}T^{b_{j}})},
$$
where $X_{i}(k_K)$ is the set of $k_K$-points satisfying $X_i$.  
\end{thm}

This theorem is furthermore applied in \cite{16} to the theory of zeta functions in group theory.
Theorem \ref{*} is shown in \cite{16} by proving general elimination of imaginaries in a language $\mathcal{L}_{\mathcal{G}}$ called the geometrical language and which expands the language of valued fields. This elimination allows one to rewrite the data in terms of classical (Denef-Pas style) uniform $p$-adic integrals, from which rationality follows uniformly in the local field.

In the appendix of \cite{16}, a more direct but similar reduction to classical $p$-adic integrals is followed, and it is this reduction that is made uniform in the local field in this paper.

An interesting aspect of Theorem \ref{*} is the appearance of the positive integer $q$ in the denominator. In more classical Poincar\'e series in this context (e.g. \cite{09}, \cite{09b}), less general denominators suffice, namely without a factor $q>0$. In this paper we use even more general denominators, namely, we may divide by the number of points on (nonempty and finite) definable subsets over the residue field. We develop a corresponding theory of $p$-adic and motivic integration, of what we call \emph{rational} motivic constructible functions (altering the notion of motivic constructible functions from \cite{04} and \cite{07}). The benefits are that we need not restrict to the semi-algebraic case and that we don't rely on elimination of imaginaries. This allows us to obtain rationality in the uniform analytic contexts from \cite{06b}, \cite{13b}, and \cite{03}, and also in the axiomatic context from \cite{07}.

Let us state our main result to indicate the more general nature of our denominators.

Let $\cT$ be a theory in a language $\cL$ extending $\LPas$. Suppose that $\cT$ has properties ($*$) and ($**$) as in section \ref{sec:ax} below (see Section \ref{sec:ex} for concrete, analytic examples of such $\cT$). Suppose for convenience here that every definable subset in the residue field sort is definable in the language of rings (this assumption is removed in the later form \ref{***} of the main theorem of the introduction).
Let $\varphi(x,y,n)$ be an $\cL$-formula with free variables $x$ running over $K^m$, $y$ running over $K^m$ and $n$ running over $\NN$.
Suppose that for each local field $K$ and each $n$, $\varphi(x,y,n)$ gives an equivalence relation $\sim_{K,n}$ on $K^m$ with finitely many, say, $a_{\varphi,K,n}$, equivalence classes.
For each local field $K$ consider the associated Poincar\'e series
$$
P_{\varphi,K}(T) =\sum_{n\geq 0}a_{\varphi,K,n}T^{n}.
$$

\begin{maintheorem}\label{**}
There exists $M>0$ such that the power series  $P_{\varphi,K}(T)$ is rational in $T$ for each local field $K$ whose residue field has characteristic at least $M$. Moreover, for such $K$, the series $P_{\varphi,K}(T)$ only depends on the residue field $k_K$ (namely, two local fields with isomorphic residue field give rise to the same Poincar\'e series).

More precisely, there exist nonnegative integers $N,M,k,b_j,e_i,$, integers $a_j$ and formulas $X_{i}$ and $Y$ in the ring language for $i=0,\ldots,N$ and $j=0,\ldots,k$, such that for each $j$, $a_{j}$ and $b_{j}$ are not both $0$, and, for  all local fields $K$ with residue field of characteristic at least $M$, $Y(k_K)$ is nonempty and
$$
P_{\varphi,K}(T) = \frac{\sum\limits_{i=0}^{N} (-1)^{e_i}  \#X_{i}(k_K)T^{i}}{\# Y(k_K) \cdot  \prod_{j=1}^{k}(1-q_K^{a_{j}}T^{b_{j}})}.
$$
\end{maintheorem}

As in \cite{16}, our theorem is related to zeta functions of groups, zeta functions of twist isoclasses of characters, the abscissa of convergence of Euler products, etc., but we do not give new applications in this direction as compared to \cite{16}.

In fact, we will give a more general theorem \ref{****} which describes the dependence of the numbers $a_{\varphi,K,n}$ on $n$ (and on completely general parameters) by means of a rational motivic constructible function.

In Section \ref{sec:defn} we recall the conditions on the language $\cL$ from \cite{07}. In Section \ref{sec:rat} we introduce rational motivic constructible functions, their motivic integrals, and their specializations to local fields. In Section \ref{sec:proof} we give some generalizations and the proofs of our main theorems.

\section{Analytic languages}\label{sec:defn}

\subsection{}

In Section \ref{sec:DP} we recall the Denef-Pas language and quantifier elimination in its corresponding theory of henselian valued fields of characteristic zero. In Section \ref{sec:ax} we develop axioms for expansions of the Denef-Pas language and its theory, following \cite{07}. In Section \ref{sec:ex} we recall that certain analytic structures satisfy the axioms from Section \ref{sec:ax}. Based on these axioms, we extend in Section \ref{sec:rat} the motivic integration from \cite{07} to a situation with more denominators.

\subsection{The language of Denef-Pas}\label{sec:DP}

Let $K$ be a valued field, with valuation map $\ord:K^{\times}\rightarrow \Gamma_K$ for some additive ordered  group $\Gamma_K$, $\cO_K$ the valuation ring of $K$ with maximal ideal $\cM_K$ and residue field  $k_K$. We denote by $x\rightarrow\overline{x}$ the projection $\cO_K\rightarrow k_K$ modulo $\cM_K$. An angular  component map (modulo $\cM_K$) on $K$ is a multiplicative map $\ac:K^{\times}\rightarrow k_K^{\times}$ extended by putting $\ac(0)=0$ and satisfying $\ac(x)=\overline{x}$ for all $x$ with $\ord(x)=0$.

The language $\LPas$ of Denef-Pas is the three-sorted language

\begin{center}
$(\mathcal{L}_{\rm ring},\mathcal{L}_{\rm ring},\mathcal{L}_{\rm oag},\ord,\ac)$
\end{center}
with as sorts:

(i) a sort $\VF$ for the valued field-sort,

(ii) a sort $\RF$ for the residue field-sort, and

(iii) a sort $\VG$ for the value group-sort,

the first copy of  $\mathcal{L}_{\rm ring}$ is used for the sort $\VF$, the second copy for $\RF$, the language $\mathcal{L}_{\rm oag}$ is the language $(+,<)$ of ordered abelian groups for $\VG$, $\ord$ denotes the valuation map on non-zero elements of $\VF$, and $\ac$ stands for an angular component map from $\VF$ to $\RF$. 

As usual for first order formulas, $\LPas$-formulas are built up from the $\LPas$-symbols together with variables, the logical connectives $\wedge$ (and), $\vee$ (or), $\neg$ (not), the quantifiers $\exists, \forall$, the equality symbol $=$, and possibly parameters (see \cite{18} for more details).

Let us briefly recall the statement of the Denef-Pas theorem on elimination of valued field quantifiers in the language $\LPas$. Denote by $H_{\ac,0}$ the $\LPas$-theory of the above described structures whose valued field is Henselian and whose residue field of characteristic zero. Then the theory $H_{\ac,0}$ admits elimination of quantifiers in the valued field sort, see \cite{22P}, Thm. 4.1 or \cite{04}, Thm. 2.1.1.

\subsection{Expansions of the Denef-Pas language: an axiomatic approach}\label{sec:ax}

In this section we single out precise axioms needed to perform motivic integration, following \cite{07}. Apart from cell decomposition, the axioms involve a Jacobian property for definable functions and a so-called property ($*$) which requires at the same time orthogonality between the value group and residue field and that the value group has no other structure than that of an ordered abelian group. Although these theories are about equicharacteristic $0$ valued fields, by logical compactness we will be able to use them for local fields of large residue field characteristic.

Let us fix a language $\mathcal{L}$ which contains  $\LPas$ and which has the same sorts as  $\LPas$. Let $\cT$ be an $\cL$-theory containing $H_{\ac,0}$. The requirements on $\cT$ will be summarized in Definition \ref{04} below.

\begin{defn}\label{00}(Jacobian property for a function). Let $K$ be  a valued field. Let $F: B\rightarrow B'$ be a function with $B, B'\subset K$. We say that $F$ has the Jacobian property if the following conditions hold all together:
\begin{itemize}
\item $F$ is a bijection and $B, B'$ are balls in $K$, namely of the form $\{x\mid \ord ( x-a) > \gamma \}$ for some $a\in K$ and $\gamma\in \Gamma_K$,
\item $F$ is $C^{1}$ on $B$ with derivative $F'$,
\item $F'$ is nonvanishing and $\ord(F')$ and $\ac(F')$ are constant on $B$,
\item for all $x,y\in B$  we have
$$\ord(F')+\ord(x-y)=\ord(F(x)-F(y))$$
and
$$\ac(F')\cdot \ac(x-y)=\ac(F(x)-F(y)).$$
\end{itemize}
\end{defn}

\begin{defn}\label{01}(Jacobian property for $\mathcal{T}$). We say that the Jacobian property holds for the $\mathcal{L}$-theory $\cT$ if for any model $\mathcal{K}$ the following holds.

Write $K$ for the $\VF$-sort of $\mathcal{K}$.
For any finite set $A$ in $\mathcal{K}$ and any $A$-definable function $F: K\rightarrow K$ there exists an $A$-definable function
$$f:K\rightarrow S$$
with $S$ a Cartesian product of sorts not involving $K$ such that each infinite fiber $f^{-1}(s)$ is a ball on which $F$ is either constant or has the Jacobian property.
\end{defn}

\begin{defn}\label{02}(Split). We say that  $\mathcal{T}$ is split if the following conditions hold for any model $\mathcal{K}$. Write $K$ for the $\VF$-sort of $\mathcal{K}$.
\begin{itemize}
\item any $\mathcal{K}$-definable subset of $\Gamma_K^{r}$ for any $r\geq 0$ is $\Gamma_K$-definable in the language of ordered  abelian groups $(+,<)$, 
\item for any finite set $A$ in $\mathcal{K}$ and any $r,s\geq 0$, any $A$-definable subset $X\subset k_K^{s}\times \Gamma_K^{r}$ is equal to a finite disjoint union of $Y_{i}\times Z_{i}$ where the $Y_{i}$ are $A$-definable subsets of $k_K^{s}$ and the $Z_{i}$ are $A$-definable subsets of $\Gamma_K^{r}$.
\end{itemize}
\end{defn}
\begin{defn}\label{03}(Finite $b$-minimality). The theory $\mathcal{T}$ is called finitely $b$-minimal if for any model $\mathcal{K}$ of $T$ the following conditions hold. Write $K$ for the $\VF$-sort of $\mathcal{K}$. Each locally constant $\mathcal{K}$-definable function $g:K^{\times}\rightarrow K$ has finite image  and for any finite set $A$ in $K$ and any $A$-definable set $X\subset K$ there exist an $A$-definable function
$$f:X\rightarrow S$$
with $S$ a Cartesian product of the form $k_K^r\times \Gamma_K^t$ for some $r,t$ and an $A$-definable function
$$c:S\rightarrow K$$
such that each nonempty fiber $f^{-1}(s)$ of $s\in S$ is either  the singleton $\lbrace c(s)\rbrace$ or the ball of the form
$$\lbrace x\in K|\ac(x-c(s))=\xi(s), \ord(x-c(s))=\eta(s)\rbrace$$
for some $\xi(s)$ in $k_K$ and some $\eta(s)\in \Gamma_K$.
\end{defn}

Recall that $\cT$ is an $\cL$-theory containing $H_{\ac,0}$, where $\cL$ contains $\LPas$ and has the same sorts  as  $\LPas$.
\begin{defn}\label{04} We say that $\mathcal{T}$ has property ($*$) if it is split, finitely $b$-minimal,  and has the Jacobian property.
\end{defn}

\begin{defn}\label{04b} We say that $\mathcal{T}$ has property ($**$) if all local fields can be equipped with $\cL$-structure and such that, for any finite subtheory $\cT'$ of $\cT$, local fields with large enough residue field characteristic are models of $\cT'$.
\end{defn}

\begin{example}\label{13}
The $\LPas$-theory $H_{\ac,0}$ of Henselian valued field with equicharacteristic $(0,0)$ has properties $(*)$ and ($**$). It even has property ($*$) in a resplendent way, namely, the theory $\cT$ in an expansion $\cL$ of $\LPas$ which is obtained from $\LPas$ by adding constant symbols from a substructure of a model of $H_{\ac,0}$ (and putting its diagram into $\cT$) and by adding any collection of relation symbols on $\RF^n$ for $n\geq 0$ has property ($*$), see \cite{07}, Thm. 3.10 or \cite{04}, Section 7 and Theorem 2.1.1.
\end{example}

Analytic examples of theories with properties ($*$) and ($**$) are given in the next section.

\subsection{Analytic Expansions of the Denef-Pas language}\label{sec:ex}


Our main example is a uniform version (on henselian valued fields) of the $p$-adic subanalytic language of \cite{11}. This uniform analytic structure is taken from \cite{06b} and is a slight generalization of the uniform analytic structure introduced by van den Dries in \cite{13b}; it also generalizes \cite{03}. While van den Dries obtained quantifier elimination results and Ax-Kochen principles, the full property ($*$) is shown in the more recent work \cite{06} and \cite{06b}; see Remark \ref{rem:comp} below for a more detailed comparison. Property ($**$) will be naturally satisfied.

Fix a commutative noetherian ring $A$ (with unit $1\not=0$) and fix an ideal $I$ of $A$ with $I\not=A$. Suppose that $A$ is complete for the $I$-adic topology. By complete we mean that the inverse limit of $A/I^n$ for $n\in \NN$ is naturally isomorphic to $A$. An already interesting example is $A=\ZZ[[t]]$ and $I=t\ZZ[[t]]$. For each $m$, write $A_m$ for
$$A[\xi_1,\ldots,\xi_m]{\,\, }{\widehat{}}\ ,
$$
namely the $I$-adic completion of the polynomial ring $A[\xi_1,\ldots,\xi_m]$, and put $\mathcal{A} = (A_{m})_{m\in\NN}$.

\begin{defn}[Analytic structure]\label{10} Let $K$ be a valued field. An analytic $\mathcal{A}$-structure on $K$ is a collection of ring homomorphisms
$$
\sigma_{m}: A_{m}\to \mbox{ ring of $\mathcal{O}_{K}$-valued functions on } \mathcal{O}_{K}^{m}
$$
for all $m\geq 0$ such that:
\begin{itemize}
\item[(1)] $I\subset \sigma^{-1}_{0}(\mathcal{M}_{K})$,

\item[(2)]  $\sigma_{m}(\xi_{i}) =$ the $i$-th coordinate function on $\mathcal{O}_{K}^{m}, i = 1, . . . , m,$
\item[(3)] $\sigma_{m+1}$ extends $\sigma_{m}$ where we identify in the obvious way functions on $\mathcal{O}_{K}^{m}$ with functions on $\mathcal{O}_{K}^{m+1}$ that do not depend on the last coordinate.
\end{itemize}
\end{defn}

Let us expand the example that $A=\mathbb{Z}[[t]]$, equipped with the $t$-adic topology.
For any field $k$, the natural $\LPas$-structure on $k((t))$ with the $t$-adic valuation has a unique $\mathcal{A}$-structure if one fixes $\sigma_{0}(t)$ (in the maximal ideal, as required by (1)). Likewise, for any finite field extension $K$ of $\QQ_p$, for any prime $p$, say, with a chosen uniformizer $\varpi_K$ of $\cO_K$ so that $\ac$ is also fixed, the natural $\LPas$-structure has a unique $\mathcal{A}$-structure up to choosing $\sigma_{0}(t)$ (in the maximal ideal).

\begin{defn}\label{11}
The $\mathcal{A}$-analytic language $\mathcal{L}_{\mathcal{A}}$ is defined as $\LPas \cup (A_{m})_{m\in\NN}$. An $\mathcal{L}_{\mathcal{A}}$-structure is an $\LPas$-structure which is equipped with an analytic $\mathcal{A}$-structure. Let $\cT_\cA$ be the theory $H_{\ac,0}$ together with the axioms of such $\mathcal{L}_{\mathcal{A}}$-structures.
\end{defn}

\begin{thm}[\cite{06b}]\label{12}
The theory $\cT_\cA$ has property $(*)$. It does so in a resplendent way (namely, also expansions as in Example \ref{13} have property ($*$) ). If $A=\mathbb{Z}[[t]]$ with ideal $I=t\mathbb{Z}[[t]]$, then it also has property ($**$) and every definable subset in the residue field sort is definable in the language of rings.
\end{thm}
\begin{proof}
By Theorem 3.2.5 of \cite{06b}, there is a separated analytic structure $\cA'$ such that $\cL_{\cA'}$ is a natural definitial expansion of $\cL_{\cA}$, with natural corresponding theory $\cT_{\cA'}$, specified in \cite{06b}. Now property ($*$) follows from Theorem 6.3.7 of \cite{06} for $\cT_{\cA'}$ (even resplendently). The statements when $A=\mathbb{Z}[[t]]$ and $I=t\mathbb{Z}[[t]]$ are clear (that every definable subset in the residue field sort is definable in the language of rings follows from quantifier elimination for $\cL_{\cA'}$ of  Theorem 6.3.7 of \cite{06}).
\end{proof}

Note that Theorem \ref{12} includes Example \ref{13} as a special case by taking $A=\ZZ$ with $I$ the zero ideal. Other examples of analytic theories that have property $(*)$ can be found in \cite{06b}, see also Section 4.4 of \cite{06}.

\begin{remark}\label{rem:comp}
Let us highlight some of the differences with the uniform analytic structure from \cite{13b}. In \cite{13b}, a variant of Definition \ref{10} of analytic $\cA$-structures is given which is slightly more  stringent, see Definition (1.7) of \cite{13b}. With this notion of (1.7), van den Dries proves quantifier elimination (resplendently) in Theorem (3.9) of \cite{13b}, which implies that the theory is split (see Definition \ref{02} above). However, more recent work is needed in order to prove the Jacobian property and finite $b$-minimality (see Definitions \ref{01} and \ref{03}), and that is done in \cite{06}, Theorem 6.3.7, for separated analytic structures. A reduction (with a definitial expansion) from an analytic $\cA$-structure (as in Definition \ref{10}) to a separated analytic structure (as in \cite{06}) is given in \cite{06b}.
\end{remark}

\section{Rational constructible  motivic functions}\label{sec:rat}

\subsection{}

We introduce rational constructible motivic functions and their motivic integrals, as a variant of the construction of motivic integration in \cite{07}. We will use this variant to prove Theorem \ref{**} and its generalizations \ref{***}, \ref{****}.

Let us fix a theory $\cT$ (in a language $\cL$) with property ($*$). From Section \ref{22} on, we will assume that $\cT$ also has property ($**$), to enable to specialize to local fields of large residue field characteristic, by logical compactness.

Up to Section \ref{subsec:rat}, we recall terminology from \cite{07}. From Section \ref{subsec:rat} on, we introduce our variant of rational constructible motivic functions.

\subsubsection{The category of definable subsets}

By a $\mathcal{T}$-field we mean a valued field $K$ with residue field $k_K$ and value group $\ZZ$, equipped with an $\cL$-structure so that it becomes a model of $\mathcal{T}$. (For set-theoretical reasons, one may want to restrict this notion to valued fields $K$ living in a very large set, or, to consider the class of all $\mathcal{T}$-fields.)

For any integers  $n,m,r\geq 0$ , we denote by $h[n,m,r]$ the functor sending a $\mathcal{T}$-field $K$ to
$$h[n,m,r](K):=K^{n}\times k_K^{m}\times\ZZ^{r}$$

Here, the convention is that $h[0,0,0]$ is the definable subset of the singleton $\{ 0\}$, i.e. $h[0,0,0](K)=\{ 0\}$.

We call a collection of subsets $X(K)$ of $h[n,m,r](K)$ for all $\cT$-fields $K$ a definable subset if there exists an $\mathcal{L}$-formula $\phi(x)$ with free variables $x$ corresponding to elements of $h[m,n,r]$ such that
$$
X(K)=\lbrace x\in h[m,n,r](K)|\phi(x) \mbox{ holds in the $\cL$-structure }(K,k_K,\ZZ)\rbrace
$$
for all $\mathcal{T}$-fields $K$.

A definable morphism $f:X\rightarrow Y$ between two definable subsets $X,Y$ is given by a definable subset $G$ such that $G(K)$ is the graph of a function $X(K)\rightarrow Y(K)$ for all $\mathcal{T}$-fields $K$.

Denote by $\Def(\mathcal{T})$ (or simply $\Def$) the category of definable subsets with definable morphisms as morphisms. If $Z$ is a definable subset, we denote by $\Def_{Z}(\mathcal{T})$ (or simply $\Def_{Z}$) the category of definable subsets $X$ with a specified definable morphism $X\rightarrow Z$; a morphism between $X,Y\in \Def_{Z}$ is a definable morphism $X\to Y$ which makes a commutative diagram with the specified morphisms $X\rightarrow Z$ and $Y\rightarrow Z$. To indicate that we work over $Z$ for some $X$ in $\Def_{Z}$, we will often write $X_{/Z}$.

For every morphism $f:Z\rightarrow Z'$ in $\Def$, by composition with $f$, we can define a functor
$$f_{!}:\Def_{Z}\rightarrow\Def_{Z'}$$
sending $X_{/Z}$ to $X_{/Z'}$. Using the fiber product, we can define a functor
$$f^{*}:\Def_{Z'}\rightarrow\Def_{Z}$$
by sending $Y_{/Z'}$ to $(Y\otimes_{Z'}Z)_{/Z}$.

When $Y$ and $Y'$ are definable sets, we write $Y\times Y'$ for their Cartesian product. We also write $Y[m,n,r]$ for the product $Y\times h[m,n,r]$.

By a point on a definable subset $X$, we mean a tuple $x=(x_{0},K)$ where $K$ is a $\mathcal{T}$-field and $x_{0}\in X(K)$. We write $|X|$ for the collection of all points that lie on $X$.
\subsubsection{Constructible Presburger functions} 
We follow \cite[Section 5, 6]{07}.
Consider a formal symbol $\mathbb{L}$ and the ring
$$
\mathbb{A}:=\mathbb{Z}[\mathbb{L},\mathbb{L}^{-1},\bigcup\limits_{i>0} \frac{1}{1-\mathbb{L}^{-i}}]
$$

For every real number $q>1$, there is a unique morphism of rings $\vartheta_{q}:\mathbb{A}\rightarrow \mathbb{R}$ mapping $\mathbb{L}$ to $q$, and it is obvious that $\vartheta_{q}$ is injective for  $q$ transcendental. Define a partial ordering on $\mathbb{A}$ by setting $a\geq b$ if for every real number with $q>1$ one has $\vartheta_{q}(a)\geq \vartheta_{q}(b)$. We denote by $\mathbb{A}_{+}$ the set $\lbrace a\in \mathbb{A}|a\geq 0\rbrace$.
\begin{defn}\label{14}
Let $S$ be a definable subset in $\Def$. The ring $\mathcal{P}(S)$  of constructible  Presburger functions  on $S$  is the subring  of the ring of functions $|S|\rightarrow \mathbb{A}$ generated by all constant functions $|S|\rightarrow \mathbb{A}$,  by all functions $\widehat{\alpha}: |S|\rightarrow\mathbb{A}$ corresponding to a definable morphism $\alpha:S\rightarrow h[0,0,1]$,  and by all functions $\mathbb{L}^{\widehat{\beta}}:|S|\rightarrow \mathbb{A}$ corresponding to a definable morphism $\beta: S\rightarrow h[0,0,1]$. We denote by $\mathcal{P}_{+}(S)$ the semi-ring consisting of functions in $\mathcal{P}(S)$ wich take values in $\mathbb{A}_{+}$. Let $\mathcal{P}_{+}^{0}(S)$ be the sub-semi-ring of $\mathcal{P}_{+}{S}$ generated by the characteristic functions $\11_{Y}$ of definable subsets $Y\subset S$ and by the constant function $\LL-1$.

If $Z\rightarrow Y$ is a morphism in $\Def$, composition with $f$ yields a natural pullback morphism $f^{*}:\mathcal{P}(Y)\rightarrow\mathcal{P}(Z)$ with restrictions $f^{*}:\mathcal{P}_{+}(Y)\rightarrow\mathcal{P}_{+}(Z)$ and
$f^{*}:\mathcal{P}_{+}^0(Y)\rightarrow\mathcal{P}_{+}^0(Z)$.
\end{defn}

\subsubsection{Rational constructible  motivic functions}\label{subsec:rat}

Definition \ref{15} is taken from \cite{04} \cite{07}. Right after this, we start our further localizations.
\begin{defn}\label{15}
Let $Z$ be a definable subset  in $\Def$. Define the semi-group $\mathcal{Q}_{+}(Z)$  as the quotient of the free abelian semigroup over symbols $[Y]$ with $Y_{/Z}$ a definable subset of $Z[0,m,0]$ with the projection to $Z$, for some $m\geq 0$, by relations
\begin{itemize}
\item[(1)] $[\emptyset\rightarrow Z]=0$,
\item[(2)] $[Y]=[Y']$ if $Y\rightarrow Z$ is isomorphic to $Y'\rightarrow Z$,
\item[(3)] $[(Y\cup Y')]+[(Y\cap Y')]=[Y]+[Y']$
for $Y$ and $Y'$ definable subsets of a common $Z[0,m,0]\rightarrow Z$ for some $m$.
\end{itemize}
The Cartesian fiber product  over $Z$ induces a natural semi-ring  structure on $\mathcal{Q}_{+}(Z)$ by  setting
\begin{center}
$[Y]\times [Y']=[Y\otimes_{Z} Y']$
\end{center}

Now let $\mathcal{Q}_{+}^{*}(Z)$ be the sub-semi-ring of $\mathcal{Q}_{+}(Z)$ given by
$$
\{[Y\overset{f}{\rightarrow}Z] \in \mathcal{Q}_{+}(Z) \mid \forall x\in Z,\ f^{-1}(x)\neq \emptyset\}.
$$
Then, $\mathcal{Q}_{+}^{*}(Z)$ is a multiplicatively closed set of $\mathcal{Q}_{+}(Z)$. So, we can consider the localization $\tilde{\mathcal{Q}}_{+}(Z)$ of $\mathcal{Q}_{+}(Z)$ with respect to $\mathcal{Q}_{+}^{*}(Z)$.

Note that if $f: Z_{1}\rightarrow Z_{2}$ is a morphism in $\Def$ then we have  natural pullback morphisms:
$$f^{*}: \mathcal{Q}_{+}(Z_{2})\rightarrow\mathcal{Q}_{+}(Z_{1})$$
by sending $[Y]\in \mathcal{Q}_{+}(Z_{2})$ to $[Y\otimes_{Z_{2}}Z_{1}]$
and
$$f^{*}: \tilde{\mathcal{Q}}_{+}(Z_{2})\rightarrow\tilde{\mathcal{Q}}_{+}(Z_{1})$$
by sending $\dfrac{[Y]}{[Y']}\in \tilde{\mathcal{Q}}_{+}(Z_{2})$ to $\dfrac{[Y\otimes_{Z_{2}}Z_{1}]}{[Y'\otimes_{Z_{2}}Z_{1}]}$.
One easily checks that these are well-defined.
We write $\LL$ for the class of $Z[0,1,0]$ in $\mathcal{Q}_{+}(Z)$, and, in $\tilde{\mathcal{Q}}_{+}(Z)$.

\end{defn}

\begin{defn}\label{16}
Let $Z$ be in $\Def$. Using the semi-ring morphism $\mathcal{P}_{+}^{0}(Z)\rightarrow\mathcal{Q}_{+}(Z)$ which sends $\11_{Y}$ to $[Y]$ and $\LL-1$ to $\LL-1$, the semi-ring $\mathcal{C}_{+}(Z)$ is defined as follows in \cite[Section 7.1]{07}:
$$\mathcal{C}_{+}(Z)=\mathcal{P}_{+}(Z)\otimes_{\mathcal{P}_{+}^{0}(Z)}\mathcal{Q}_{+}(Z).$$
Elements of $\mathcal{C}_{+}(Z)$ are called (nonnegative) constructible motivic functions on $Z$.
In the same way, we define the semi-ring of rational (nonnegative) constructible motivic functions as
$$\tilde{\mathcal{C}}_{+}(Z)=\mathcal{P}_{+}(Z)\otimes_{\mathcal{P}_{+}^{0}(Z)}\tilde{\mathcal{Q}}_{+}(Z),$$
by using the semi-ring morphism $\mathcal{P}_{+}^{0}(Z)\rightarrow \tilde{\mathcal{Q}}_{+}(Z)$ which sends $\11_{Y}$ to $[Y]$ and $\LL-1$ to $\LL-1$.

If $f: Z\rightarrow Y$ is a morphism in $\Def$ then there is a natural pullback morphism from \cite[Section 7.1]{07}:
$$f^{*}: \mathcal{C}_{+}(Y)\rightarrow\mathcal{C}_{+}(Z)$$
sending $a\otimes b$ to $f^{*}(a)\otimes f^{*}(b)$, where $a\in\mathcal{P}_{+}(Y)$ and $b\in\mathcal{Q}_{+}(Y)$.
Likewise, we have the pullback morphism:
$$f^{*}: \tilde{\mathcal{C}}_{+}(Y)\rightarrow\tilde{\mathcal{C}}_{+}(Z)$$
sending $a\otimes b$ to $f^{*}(a)\otimes f^{*}(b)$, where $a\in\mathcal{P}_{+}(Y)$ and $b\in\tilde{\mathcal{Q}}_{+}(Y)$.
\end{defn}

Since $\mathcal{T}$ is split, the canonical morphism
\begin{equation}\label{eq:otimes}
\mathcal{P}_{+}(Z[0,0,r])\otimes_{\mathcal{P}_{+}^{0}(Z)}\tilde{\mathcal{Q}}_{+}(Z[0,m,0])\rightarrow\tilde{\mathcal{C}}_{+}(Z[0,m,r])
\end{equation}
is an isomorphism of semi-rings, where the homomorphisms $p^{*}:\mathcal{P}_{+}^{0}(Z)\rightarrow\mathcal{P}_{+}(Z[0,0,r])$ and $q^{*}:\mathcal{P}_{+}^{0}(Z)\rightarrow\tilde{\mathcal{Q}}_{+}(Z[0,m,0])$ come from the pullback homomorphism of the two projections $p:Z[0,0,r]\rightarrow Z$ and $q:Z[0,m,0]\rightarrow Z$, similar to \cite[Proposition 7.5]{07}.

\begin{prop}\label{prop:rat}
For $F$ in $\tilde{\mathcal{C}}_{+}(\ZZ^r)$ there exist $[Y]$ in $\cQ_+(h[0,0,0])$ and $G\in \cC_+(\ZZ^r)$ such that one has the equality
$$
\overline{G} = \overline{[Y]}\cdot F
$$
in $\tilde{\mathcal{C}}_{+}(\ZZ^r)$, with $\overline{G}$ the image of $G$ under the natural map
${\mathcal{C}}_{+}(\ZZ^r)\to \tilde{\mathcal{C}}_{+}(\ZZ^r)$ and similarly for $\overline{[Y]}$.
\end{prop}
\begin{proof}
This follows directly from the isomorphism from (\ref{eq:otimes}), the fact that $\cT$ is split, and the definition of $\tilde{\mathcal{Q}}_{+}(h[0,0,0])$.
\end{proof}

\subsection{Integration of rational constructible motivic functions}
In the three next sections, we give the definition of rational constructible motivic functions and their integrals, which follows the same construction as for integration of constructible motivic functions in \cite{07}, similarly using property ($*$).


\subsubsection{Integration over the residue field}\label{17}

We adapt \cite[Section 6.2]{07} to our setting.
Suppose that $Z\subset X[0,k,0]$ for some  $k\geq 0$ and $a\in \tilde{\mathcal{Q}}_{+}(Z)$, we write $a=\dfrac{[Y]}{[Y']}$ for some $[Y\overset{f}{\rightarrow}Z]\in \mathcal{Q}_{+}(Z)$ and $[Y'\overset{f'}{\rightarrow}Z]\in \mathcal{Q}_{+}^{*}(Z)$. We write $\mu_{/X}$ for the corresponding formal integral in the fibres of the coordinate projection $Z\rightarrow X$
$$\mu_{/X}:\tilde{\mathcal{Q}}_{+}(Z)\rightarrow\tilde{\mathcal{Q}}_{+}(X), \frac{[Y]}{[Y']}\mapsto \frac{[Y]}{[Y'']}$$
where $[Y'']=[Y'\bigsqcup(X\backslash {\rm Im} f')]$, where $\bigsqcup$ denotes the disjoint union (a disjoint union of definable sets can be realized as a definable set by using suitable piecewise definable bijections). Note that $Y''$ is built from $Y$ by using that the class of a definable singleton is the multiplicative unit to preserve the property that the fibers over $Z$ are never empty.
\subsubsection{Integration over the value group}\label{18}

This section follows \cite[Section 5]{07} and will be combined with the integration from section \ref{17} afterwards.
Let $Z\in \Def$ and $f\in\mathcal{P}(Z[0,0,r])$.  For any $\mathcal{T}$-field $K$ and any $q>1$ we write $\vartheta_{q,K}(f):Z(K)\rightarrow \RR$ for the function sending $z\in Z(K)$ to $\vartheta_{q}(f(z,K))$.
Then $f$ is called $Z$-integrable if for each $\mathcal{T}$-field $K$, each $q>1$ and for each $z\in Z(K)$, the family $(\vartheta_{q,K}(f)(z,i))_{i\in \ZZ^{r}}$ is summable.
The collection of $Z$-integrable functions in $\mathcal{P}(Z[0,0,r])$ is denoted by $I_{Z}\mathcal{P}(Z[0,0,r])$ and $I_{Z}\mathcal{P}_{+}(Z[0,0,r])$ is the collection of $Z$-integrable functions in $\mathcal{P}_{+}(Z[0,0,r])$.

We recall from \cite[Theorem-Definition 5.1]{07} that for each $\phi\in I_{Z}\mathcal{P}(Z[0,0,r])$, there exists a unique function $\varphi:= \mu_{/Z}(\phi)$ in $\mathcal{P}_{Z}$ such that for all $q>1$, all $\mathcal{T}$-fields $K$, all $z\in Z(K)$, one has
$$\vartheta_{q,K}(\varphi)(z)=\sum_{i\in\ZZ^{r}}\vartheta_{q,K}(\phi)(z,i)$$
and the mapping $\phi\mapsto\mu_{/Z}(\phi)$ yields a morphism of $\mathcal{P}(Z)$-modules
$$\mu_{/Z}:I_{Z}\mathcal{P}(Z\times\ZZ^{r})\rightarrow\mathcal{P}(Z).$$

\subsubsection{Integration over one valued field variable}\label{19}
We first follow \cite[Section 8]{07} and then use it for  our setting of rational motivic constructible functions in Lemma-Definiton \ref{20} below.
For a ball $B=a+b\mathcal{O}_{K}$ and any real number $q>1$, we call $\vartheta_{q}(B):=q^{-\ord b}$ the $q$-volume of $B$.  A finite or countable collection of disjoint balls in $K$, each with different $q$-volume  is  called a step-domain; we will identify a step-domain $S$ with the union of the balls in $S$. Recall from \cite{07} that a nonnegative real valued function $\varphi: K\rightarrow\RR_{\geq 0}$ is a step-function if there exists a unique step-domain $S$ such that $\varphi$ is constant and nonzero on each ball of $S$ and zero outside $S\cup\{a\}$ for some $a\in K$.

Let $q>1$ be a real number. A step-function $\varphi:K\rightarrow\RR_{\geq 0}$ with step-domain $S$ is $q$-integrable over $K$ if and only if
$$\sum_{B\in S}\vartheta_{q}(B).\varphi(B)<\infty$$

Suppose that $Z=X[1,0,0]$ for some $X\in \Def$ and $\varphi\in\mathcal{P}_{+}(Z)$. We call $\varphi$ an $X$-integrable family of step-functions if for each $\mathcal{T}$-field $K$, for each $x\in X(K)$ and for each $q>1$, the function
$$\vartheta_{q,K}(\varphi)(x,.):K\rightarrow\RR_{\geq 0}, t\mapsto\vartheta_{q,K}(\varphi)(x,t),$$
is a step-function which is $q$-integrable over $K$. For such $\varphi$ there exists a unique function $\phi$ in $\mathcal{P}_{+}(X)$ such that $\vartheta_{q,K}(\phi)(x)$ equals the $q$-integral over $K$ of $\vartheta_{q,K}(\varphi)(x,.)$ for each $\mathcal{T}$-field $K$, each $x\in X(K)$ and each $q>1$; we write
$$
\mu_{/X}(\varphi):=\phi,
$$ 
the integral of $\varphi$ in the fibers of $Z\rightarrow X$.

\begin{def-lem}\label{20}
Let $\varphi\in \tilde{\mathcal{C}}_{+}(Z)$ and suppose that $Z=X[1,0,0]$. We say that $\varphi$ is $X$-integrable  if there exists a $\phi$ in $\mathcal{P}_{+}(Z[0,m,0])$ with $\mu_{/Z}(\phi)=\varphi$ such that $\phi$ is $X[0,m,0]$-integrable  and then
$$
\mu_{/X}(\varphi):=\mu_{/X}(\mu_{/X[0,m,0]}(\phi))\in\tilde{\mathcal{C}}_{+}(X)$$
is independent  of the choices and is called the integral of $\varphi$ in the fibers of $Z\rightarrow X$.
\end{def-lem}
\begin{proof}
This follows using property ($*$) of $\mathcal{T}$ in the same way as in lemma-definition 8.2 of \cite{07}. \end{proof}

\subsubsection{Integration of rational constructible motivic functions in the general case}\label{21}

Combining the three cases above, we define integrability and the integral $\mu_{/X}(\varphi)$ of an integrable rational constructible motivic function $\varphi\in \tilde{\mathcal{C}}_{+}(X[m,n,r])$ by Tonelli-Fubini iterated integration in a similar way as in Lemma-Definition 9.1 of \cite{07}. More precisely, we will define the integrals in the fibers of a general coordinate projection $X[n,m,r]\rightarrow X$ by induction on $n\geq 0$.
\begin{def-lem}Let $\varphi$ be in $\tilde{\cC}_{+}(Z)$ and suppose that $Z=X[n,m,r]$ for some $X$ in $\Def$.

If $n=0$ we say that $\varphi$ is $X$-integrable is and only if $\varphi$ is $X[0,m,0]$-integrable. If this holds then
$$\mu/_{X}(\varphi):=\mu_{/X}(\mu_{/X[0,m,0]}(\varphi))\in \tilde{\cC}_{+}(X)$$
is called the integral of $\varphi$ in the fibers of $Z\rightarrow X$.

If $n\geq 1$, we say that $\varphi$ is $X$-integrable  if there exists a definable subset $Z'\subset Z$ whose complement in $Z$ has relative dimention $<n$ over $X$ such that $\varphi':=\11_{Z'}\varphi$ is $X[n-1,m,r]$-integrable and $\mu_{/X[n-1,m,r]}(\varphi')$ is $X$-integrable. If this holds then
$$ \mu/_{X}(\varphi):=\mu_{/X}(\mu_{/X[n-1,m,r]}(\varphi'))\in \tilde{\cC}_{+}(X)$$
does not depend on the choices and is called the integral of $\varphi$ in the fibers of $Z\rightarrow X$.

Slightly more generally, let $\varphi\in \tilde{\cC}_{+}(Z)$ and suppose that $Z\subset X[n,m,r]$. We say that $\varphi$ is $X$-integrable if the extension by zero of $\varphi$ to a function $\tilde{\varphi}\in \tilde{\cC}_{+}(X[n,m,r])$ is $X$-integrable and we define $\mu_{/X}(\varphi)$ as $\mu_{/X}(\tilde{\varphi})$.
\end{def-lem}
\begin{proof}
Since $\cT$ has property $(*)$ the proof is similar to the proof for Lemma-Definition 9.1 of \cite{07}. 
\end{proof}

Based on (\ref{eq:otimes}) and the definition of integrability, for each $Z$-integrable function $\varphi\in\tilde{\mathcal{C}}_{+}(Z[0,m,r])$, one can write  $\varphi=\sum_{i} a_{i}\otimes b_{i}$, where $a_{i}\in I_{Z}\mathcal{P}_{+}(Z[0,0,r])$ and $b_{i}\in \tilde{\mathcal{Q}}_{+}(Z[0,m,0])$ and
$$\mu_{/Z}(\varphi)=\sum_{i}\mu_{/Z}(a_{i})\otimes \mu_{/Z}(b_{i}).
$$

\subsection{Interpretation of rational constructible motivic functions in non-archimedean local fields}\label{22}

In this section we show how rational constructible motivic functions can be specialized to real valued functions on local fields of large residue field characteristic, in the spirit of the specializations in \cite{08} and Proposition 9.2 of \cite{07}. Importantly, taking motivic integrals combines well with this specialization and integration over the local fields.

Let $\cT$ be a theory in a language $\cL$ extending $\LPas$. Suppose that $\cT$ has properties ($*$) and ($**$) from section \ref{sec:ax}.
For a definable set $X$, a definable function $f$ and a rational motivic constructible function $\phi$, the objects $X_K=X(K)$, $f_K$ and $\phi_K$ make sense for every local field $K$ with large residue field characteristic. We make this explicit for rational motivic constructible functions, where we assume $K$ to be a local field with large residue field characteristic (depending on the data).


\begin{itemize}
\item For $a\in \mathcal{P}_{+}(X)$, we get $a_{K}:X_K\rightarrow \QQ_{\geq 0}$ by replacing $\LL$ by $q_K$.
\item For $b=[Y]$ with $Y$ a definable subset of $X[0,m,0]$ in $\mathcal{Q}_{+}(X)$, if we write $p:Y\rightarrow X$ for the projection, one defines $b_{K}:X_K\rightarrow\QQ_{\geq 0}$ by sending $x\in X_K$ to $\#(p_K^{-1}(x))$
\item In the same way, for $b=\dfrac{1}{[Y]}$ with $Y$ a subset  of $X[0,m,0]$ in $\mathcal{Q}_{+}^{*}(X)$ and projection $p:Y\rightarrow X$, one defines $b_{K}:X_K\rightarrow\QQ_{\geq 0}$ by sending $x\in X_K$ to $\dfrac{1}{\#(p_K^{-1}(x))}$.
\item For $\phi\in\mathcal{C}_{+}(X)$ or $\phi\in\tilde{\mathcal{C}}_{+}(X)$, writing $\phi$ as a finite sum $\sum_{i}a_{i}\otimes b_{i}$ with $a_{i}\in\mathcal{P}_{+}(X)$ and $b_{i}\in \mathcal{Q}_{+}(X)$ or $b_{i}\in\tilde{\mathcal{Q}}_{+}(X)$, we get the function
$$\phi_{K}:X_K\rightarrow\QQ_{\geq 0}, x\mapsto\sum_{i}a_{iK}(x).b_{iK}(x),$$
which does not depend on the choices made for $a_{i}$ and $b_{i}$.
\end{itemize}

Taking motivic integrals commutes with taking specializations, as follows.

\begin{prop}\label{prop:int}
Let $\varphi$ be an $X$-integrable rational constructible motivic function in $\tilde{\mathcal{C}}_{+}(X[m,n,r])$ and let $\mu_{/X}(\varphi)$ be its motivic integral, in the fibers of the projection $X[m,n,r]\to X$. Then there exists $M>0$ such that for all local fields $K$ whose residue field has characteristic at least $M$ one has for each $x\in X_K$
$$
\big(\mu_{/X}(\varphi)\big)_K (x)   =  \int_{y\in K^m\times k_K^n\times \ZZ^r} \varphi_K(x,y) |dy|
$$
where one puts the normalized Haar measure on $K$, the counting measure on $k_K$ and on $\ZZ$, and the product measure $|dy|$ on  $K^m\times k_K^n\times \ZZ^r$.
\end{prop}
\begin{proof}
This follows naturally from the corresponding result for $\cC_+$ instead of $\tilde{\mathcal{C}}_{+}$ (see \cite{08} and Proposition 9.2 of \cite{07}), and the concrete definition of $\tilde{\mathcal{C}}_{+}$. \end{proof}

\section{The uniform rationality for Poincar\'e series of definable equivalence relations}\label{sec:proof}

\subsection{}
We will prove Theorem \ref{***}, which is a slight generalization of the Main Theorem \ref{**}.

Let $\cT$ be a theory in a language $\cL$ extending $\LPas$. Suppose that $\cT$ has properties ($*$) and ($**$) from section \ref{sec:ax}.
Let $\varphi(x,y,n)$ be an $\cL$-formula with free variables $x$ running over $K^m$, $y$ running over $K^m$ and $n$ running over $\NN$.
Suppose that for each local field $K$ and each $n$, $\varphi(x,y,n)$ gives an equivalence relation $\sim_{K,n}$ on $K^m$ with finitely many, say, $a_{\varphi,K,n}$, equivalence classes. (The situation that $\varphi(x,y,n)$ gives an equivalence relation on a definable subset $X_{K,n}$ of $K^m$ for each $K$ and each $n$ is similar.)
For each local field $K$ put
$$
P_{\varphi,K}(T) =\sum_{n\geq 0}a_{\varphi,K,n}T^{n}.
$$

\begin{thm}\label{***}
There exists $M>0$ such that the power series  $P_{\varphi,K}(T)$ is rational in $T$ for each local field $K$ whose residue field has characteristic at least $M$. Moreover, for such $K$, the series $P_{\varphi,K}(T)$ only depends on the $\cL$-structure induced on the residue field sort $k_K$.

More precisely, there exist nonnegative integers $N,k,b_j,e_i$, integers $a_j$ and $\cL$-formulas $X_{i}$ and $Y$ for subsets of some power of the residue field for $i=0,\ldots,N$ and $j=0,\ldots,k$, such that for each $j$, $a_{j}$ and $b_{j}$ are not both $0$, and, for  all local fields $K$ with residue field of characteristic at least $M$, $Y(k_K)$ is nonempty and
$$
P_{\varphi,K}(T) = \frac{\sum\limits_{i=0}^{N} (-1)^{e_i}  \#X_{i}(k_K)T^{i}}{\# Y(k_K) \cdot  \prod_{j=1}^{k}(1-q_K^{a_{j}}T^{b_{j}})}.
$$
\end{thm}

In fact, Theorem \ref{***} is a consequence of the following more versatile theorem:

\begin{maintheorem}\label{****}
Suppose that $\cT$ has properties ($*$) and ($**$). Let $\varphi(x,y,z)$ be an $\cL$-formula with free variables $x$ running over $\VF^n$, $y$ running over $\VF^n$ and $z$ running over an arbitrary $\cL$-definable set $Z$. Let $R$ be a definable subset of $\VF^n\times Z$.
Suppose that for each local field $K$, and each $z\in Z_K$, $\varphi(x,y,z)$ gives an equivalence relation $\sim_{K,z}$ on $R_{K,z}:=\{x\in K^n\mid (x,z) \in R_K\}$ with finitely many, say, $a_{\varphi,K,z}$, equivalence classes.
Then there exist a rational motivic function $F$ in $\tilde{\mathcal{C}}_{+}(Z)$ and a constant $M>0$ such that for each local field $K$ whose residue field has characteristic at least $M$ one has
$$
a_{\varphi,K,z} = F_K(z).
$$
\end{maintheorem}

\begin{proof}[Proof of Theorem \ref{***}]
Theorem \ref{***} follows from Theorem \ref{****} by Proposition \ref{prop:rat}, the fact that $\cT$ is split, and the rationality result Theorem 7.1 of \cite{22P}.
\end{proof}

Before giving the proof of Theorem \ref{****} we give a few more definitions and lemmas.

\subsection{Multiballs, multiboxes, and their multivolumes}

We give definitions which are inspired by concepts of the appendix of \cite{16}.

Fix a local field $K$. Recall that $q_K$ stands for the number of elements of the residue field $k_K$ of $K$. We implicitely use an ordering of the coordinates on $\cO_K^n$ in the following definition.
\begin{defn}\label{23}
Let $n\geq 1$, $0\leq r_{i}\leq +\infty$ for $i=1,...,n$ and let a nonempty set $Y\subset\mathcal{O}_{K}^{n}$ be given.

If $n=1$ then $Y$ is called a multiball of multivolume $q_K^{-r_{1}}$ if  $\Vol(Y)=q_K^{-r_{1}}$ and $Y$ is a singleton in the case that $r_{1}=+\infty$ and $Y$ is a ball in the case that $r_{1}\neq +\infty$. Here the volume $\Vol$ is taken for the Haar measure on $K$ such that $\cO_K$ has measure $1$ and we consider $q_K^{-\infty}$ to be zero.

If $n\geq 2$, then $Y$ is called a multiball of multivolume $(q_K^{-r_{1}},...,q_K^{-r_{n}})$ if and only $Y$ is of the form
\begin{center}
$\lbrace (x_{1},...,x_{n})\mid  (x_{1},...,x_{n-1})\in A,\  x_{n}\in B_{x_{1},...,x_{n-1}}\rbrace,$
\end{center}
\begin{flushleft}
where $A\subset \mathcal{O}_{K}^{n-1}$ is a multiball of multivolume $(q_K^{-r_{1}},...,q_K^{-r_{n-1}})$, $B_{x_{1},...,x_{n-1}}$ is a subset of $\mathcal{O}_{K}$ which may depend on $(x_{1},...,x_{n-1})$ with $\Vol(B_{x_{1},...,x_{n-1}})=q_K^{-r_{n}}$ and $B_{x_{1},...,x_{n-1}}$ is either a ball or a singleton. The multivolume of a multiball $Y$ is denoted by $multivol(Y)$.
\end{flushleft}
\end{defn}
\begin{defn}\label{24} We put the inverse lexicographical order (the colexicographical order) on $\RR^{n}$, namely, $(a_{1},...,a_{n})>(b_{1},...,b_{n})$ if and only if there exists $1\leq k\leq n$ such that $a_{i}=b_{i}$ for all $i>k$ and $a_{k}>b_{k}$. By this order, we can compare multivolumes.
Let $X\subset \mathcal{O}_{K}^{n}$. The multibox of $X$, denoted by $MB(X)$, is the union of the multiballs $Y$ contained in $X$ and with maximal multivolume (for the colexicographical ordering), where maximality is among all multiballs contained in $X$. We write $multivol(X)$ for $multivol(Y)$ for any multiball $Y$ contained in $X$ with maximal multivolume.
\end{defn}

Note that taking $MB$ and taking projections does not always commute and it may be that $p(MB(X))$ and $MB(p(X))$ are different, say, with $p:\cO_K^n\to \cO_K^{n-1}$ the coordinate projection to the first $n-1$ coordinates. 

\begin{defn} \label{25} Fix $1\leq m\leq n$, a set $X\subset \mathcal{O}_{K}^{n}$, and $x=(x_{1},...,x_{n})$ in $MB(X)$. Set $x_{\leq m}=(x_{1},...,x_{m})$ and let $X(m)$ be the image of $MB(X)$ under the projection from $\mathcal{O}_{K}^{n}$ to $\mathcal{O}_{K}^{m}$. Denote by $X(x,m)$ the fiber of $X(m)$ over $x_{\leq m-1}$. We write
$$multinumber_{m}(X,x)$$
for the number of balls with maximal volume contained in $X(x,m)$. Write
$$
Multinumber_{m}(X,x)
$$
for the number of balls $B$ of minimal volume with $B\cap X(x,m)\neq\emptyset$ and with $B\nsubseteq X(x,m)$. 
\end{defn}

Note that the number of balls in $\cO_K$ of any fixed volume is automatically finite.

\subsection{Definable equivalence relations }

From now on we suppose that $\cT$ has properties ($*$) and ($**$).
Let $X$ be a definable set. 
Of course, $multinumber_{m}(X_K,x)$ and $Multinumber_{m}(X_K,x)$ may vary with $K$ and $x$. In fact, 
it is difficult to give a uniform in $K$ estimate for $multinumber_{m}(X_K,x)$ but for $Multinumber_{m}(X_K,x)$ we can do it, even in definable families, see Lemma \ref{28a} and its corollary.
\begin{lem}\label{28a} 
Let $Z$ and $X\subset \VF\times Z$ be $\cL$-definable such that $X_K\subset \cO_K\times Z_K$ for all local fields $K$ of large residue field characteristic. Then
there exist positive integers $M$ and $Q$ such that, for all local fields $K$ with residue field characteristic at least $M$ and for all $z\in Z_K$, one has
$$N_{K,z}\leq Q,$$
where $N_{K,z}$ is the number of balls $B$ of minimal volume with
$$
B\cap MB(X_{K,z})\neq\emptyset\mbox{ and } B\nsubseteq MB(X_{K,z}),
$$ and where $X_{K,z}$ is the set $\{x\in \cO_K\mid (x,z)\in X_K\}$.
\end{lem}
\begin{proof}
Write $X'$ for the definable subset of $X$ such that $X'_{K,z} = MB(X_{K,z})$ for each $K$ with large residue field characteristic and each $z\in Z_K$.
 Since $\mathcal{T}$ has properties $(*)$, $(**)$, and by logical compactness, there exist positive integers $M$, a Cartesian product $S$ of sorts not involving the valued field sort and $\mathcal{L}$-definable functions $f$, $c$, $\xi$, $t$, such that for all local fields $K$ with residue field has characteristic at least $M$, we have
\begin{itemize}
\item $f_{K}:X'_{K} \rightarrow Z_K\times S_K$ is a function over $Z_{K}$ (meaning that $f_{K}$ makes a commutative diagram with the projections $Z_K\times S_K\to Z_K$ and $X'_K\to Z_K$);
\item $c_{K}: Z_K\times S_K \rightarrow \mathcal{O}_{K}$, $\xi_{K}: Z_K\times S_K\rightarrow k_{K}$,  and $t_{K}: Z_K\times S_K \rightarrow \ZZ$ are such that each nonempty fiber $f_{K}^{-1}(z,s)$, for $(z,s)\in Z_K\times S_K$, is either the singleton
    $$\{c_{K}(z,s)\}$$
     or the ball
    $$\{y\in\mathcal{O}_{K}|\ac(y-c_{K}(z,s))=\xi_{K}(z,s),\ \ord(y-c_{K}(z,s))=t_{K}(z,s)\}.
    $$
\end{itemize}

One derives from property ($*$) and compactness (as in \cite{07}) that there exists an integer $Q_0>0$ such that for for all local fields $K$ with large residue field characteristic, for each $z\in Z_K$, the range of $s\mapsto c_K(z,s)$ has no more than $Q_0$ elements. We will show that we can take $Q=Q_0$. 

Suppose first that $X'_{K,z}$ is a disjoint union of balls of volume $q_{K}^{-\alpha(K,z)}$ where $$
q_{K}^{-\alpha(K,z)}=multivol(X_{K,z}).
$$
Choose a ball $B$ with volume $q_{K}^{-\alpha(K,z)+1}$ and with  $ B\cap X'_{K,z}\neq\emptyset$, fix $y\in B\cap X'_{K,z}$ and write $f_{K}(y)=(z,s)$, so that $y$ belongs to the ball
$$
B'=\{v\in\mathcal{O}_{K}|\ac(v-c_{K}(z,s))=\xi_{K}(z,s),\ \ord(v-c_{K}(z,s))=t_{K}(z,s)\},
$$
which has volume $q_{K}^{-t_{K}(z,s)-1}$. Since $X'_{K,z}=MB(X_{K,z})$, the ball
$$ 
B(y,q_{K}^{-\alpha(K,z)})
$$
around $y$ of volume $q_{K}^{-\alpha(K,z)}$ is a maximal ball contained in $X_{K,z}$. Hence, $y\in B'\subset X_{K,z}$ implies $B'\subset B(y,q_{K}^{-\alpha(K,z)})$. It follows that
$$
t_{K}(z,s)+1\geq \alpha(K,z),
$$
which proves that
\begin{equation}\label{eq:y}
\ord(y-c_{K}(z,s))\geq \alpha(K,z)-1.
\end{equation}
Since $B$ has volume $q_{K}^{-\alpha(K,z)+1}$ and contains $y$, the inequality (\ref{eq:y}) implies that $c_{K}(z,s)\in B$ and it follows that $N_{K,z}\leq Q_{0}$.

If $X'_{K,z}$ is not a union of balls then it is contained in the range of $c_{K}(z,.):s\mapsto c_K(z,s)$, and thus also in this case we find $0=N_{K,z}\leq Q_{0}$. This shows we can take $Q=Q_0$. 
\end{proof}

\begin{cor}\label{28b}
Let $Z$ and $X\subset \VF^n\times Z$ be $\cL$-definable such that $X_K\subset \cO_K^n\times Z_K$ for all local fields $K$ of large residue field characteristic. Fix $m$ with $1\leq m\leq n$. Then
there exist positive integers $M$ and $Q$ such that
$$
Multinumber_{m}(X_{K,z},x)  < Q
$$
for all local fields $K$ with residue field characteristic at least $M$, for all $z\in Z_K$ and all $x\in \cO_K$ with $x\in MB(X_{K,z})$, where $X_{K,z}$ is the set $\{x\in \cO_K^n\mid (x,z)\in X_K\}$.
\end{cor}

\begin{proof}
The corollary follows from Lemma \ref{28a} since the condition $x\in MB(X_{K,z})$ is an $\mathcal{L}$-definable condition on $(x,z)$.
\end{proof}

%
%

\begin{lem}\label{30}
Let $Z$ and $X\subset \VF^n\times Z$ be $\cL$-definable such that $X_K\subset \cO_K^n\times Z_K$ for all local fields $K$ of large residue field characteristic.
Then there exist $M>0$ and a definable function $f:Z\rightarrow (\VG\cup\lbrace +\infty\rbrace)^{n}$ such that $$
(q_{K}^{-f_{i,K}(z)})_{i=1}^{n} = multivol(X_{K,z})
$$
for each $z\in Z_K$ and each local field $K$ with residue field characteristic at least $M$, where $X_{K,z}$ is the set $\{x\in \cO_K^n \mid (x,z)\in X_K\}$.
\end{lem}
\begin{proof}
The lemma follows easily by the definability of multiboxes and of the valuative radius of the balls involved.
\end{proof}

\begin{proof}[Proof of the main theorem \ref{****}]
Firstly, we can suppose that the sets $R_{K,z}$ for $z\in Z(K)$ are subsets of $\mathcal{O}_{K}^{n}$, up to increasing $n$ and mapping a coordinate $w\in K$ to $(w,0)\in\mathcal{O}_{K}^{2}$ if $\ord(w)\geq 0$ and to $(0,w^{-1})\in\mathcal{O}_{K}^{2}$ if $\ord(w)<0$.

By Lemma \ref{30}, there is a definable function $f_{K,z}:\mathcal{O}_{K}^{n}\rightarrow (\mathbb{Z}\cup\lbrace +\infty\rbrace)^{n}$ such that
$$
(q_{K}^{-f_{K,z,i}(x)})_{i=1}^{n} = multivol(x/_{\sim_{K,z}}),
$$
where $x/_{\sim_{K,z}}$ denotes the equivalence class of $x$ under the equivalence relation $\sim_{K,z}$. 

For each subset $I$ of $\lbrace 1,...,n\rbrace$, we set $R_{K,z,I}=\lbrace x\in R_{K,z}| f_{K,z,i}(x)< +\infty \Leftrightarrow i \in I\rbrace$. So, $R_{K,I}=(R_{K,z,I})_{z\in Z(K)}$ is defined by an $\cL$-formula for each $I$, uniformly in $z$. It is easy to see that $R_{K,z}=\coprod_{I\subset\lbrace 1,...,n\rbrace} R_{K,z,I}$ and  $(x_{\sim_{K,z}}y)\wedge (x\in R_{K,z,I})\Rightarrow y\in R_{K,z,I}$. So if we set $a_{K,z,I}=R_{K,z,I}/_{\sim_{K,z}}$ then $a_{K,z}=\sum_{I\subset\lbrace 1,...,n\rbrace} a_{K,z,I}$. The proof will be followed by the claim for each $I\subset\{1,...,n\}$. We will consider two cases.

{\bf Case 1: $I$ is the empty set.}

Because of the definition of $R_{K,I}$, we deduce that for each $z\in Z(K)$ the definable set $R_{K,z,I}$ will be a finite set.
From the proof of lemma \ref{28b} we have that  $a_{K,z,I}$ must be bounded uniformly on all local field $K$ with large residue field characteristic and $z\in Z(K)$. We can assume that $a_{K,z,I}\leq Q$ for all $z\in Z(K)$ and $\charac(k_{K})>M$. For each $1\leq d\leq Q$, we see that the condition $a_{K,z,I}=d$ will be  an $\cL$-definable condition  in $K, z$. Write 
$$
Z_{d}(K)=\lbrace z\in Z(K)|a_{K,z,I}=d\rbrace
$$
and set
$$
F_{K}(z)=\sum_{i=1}^{n}d\11_{Z_{d}},
$$
where $\11_{Z_{d}}$ stands for the characteristic function of $Z_{d}$. 
Then it is obvious that $F_{K}(z)=a_{K,z,I}$ and that we can ensure that $F\in\tilde{C}_{+}(Z)$. 

{\bf Case 2: $I$ is not the empty set.} 

 By a similar argument as for Case 1, we may and do suppose that $I=\lbrace 1,...,n\rbrace$. To simplify the notation, $R_{K,I}$ will be rewritten as $R_{K}$. 

The claim does not change if we remove from $R_{K,z}$ any $x$ with $x\notin MB(x/_{\sim_{K,z}})$, and thus,  we can assume that  
$$
R_{K,z}=\cup_{x\in R_{K,z}}MB(x/_{\sim_{K,z}}).
$$ 

Consider a definable subset $D_{K}$ of $R_{K}\times k_{K}^{n}$, described as follows:  

For each $z\in Z_K$ and each $x=(x_{1},...,x_{n})\in R_{K,z}$, the fiber $D_{K,x,z}$ of $D_K$ over $x$ and $z$ is
$$
\lbrace (\xi_{1},...,\xi_{n})\in k_{K}^{n}| \wedge_{m=1}^n \big( (\xi_{m}=0)\vee (\exists y\in x/_{\sim_{K,z}}(x,m)
$$
$$
 [(\ord(x_{m}-y)=f_{K,z,m}-1)\wedge (\xi_{m}=\overline{ac}(x_{m}-y))]) \big) \rbrace.
$$

From the definition of $D$ we observe that
$$
\# (D_{K,x,z})=\prod_{m=1}^{n}\# (D_{K,x,z,m})
$$
where $D_{K,z,x,m}$ the image under the projection of $D_{K,z,x}$ to the $m$-th coordinate. 
Moreover,  if $B$ is a ball of volume $q_{K}^{-f_{K,z,m}+1}$ such that $x_{m}\in B$ then $B\cap x/_{\sim_{K,z}}(x,m)$ is disjoint union of $\# (D_{K,z,x,m})$ balls of volume $q_{K}^{-f_{K,z,m}}$. It follows that  
$$
\# (D_{K,z,x,m})=\# (D_{K,z,x',m})
$$
for all $x$, $x'$ with $(x'\sim_{K,z}x)\wedge(x_{\leq m-1}=x'_{\leq m-1})\wedge (x'_{m}\in B)$ and so we can denote this number by $N_{x/_{\sim_{K,z}},x_{\leq m-1},B}$.

By Lemma \ref{30}, the function  $f_{K}(x,z)=f_{K,z}(x)=\sum_{i=1}^{n}f_{K,z,i}(x)$ is an $\cL$-definable function  from $R_{K}$ to $\mathbb{Z}$. By Corollary \ref{28b}, there exist $M$ and  $Q$ in 
$\mathbb{N}$ such that 
$$
\prod_{m=1}^{n}Multinumber_{m}(x/_{\sim_{K,z}},x)\leq Q
$$
for every  $(x,z)\in R_{K}$ and all $K$ with $\charac(k_{K})>M$. Hence, for each $d$ with $1\leq d\leq Q$ the set
$$
R_{K}(d):=\lbrace (x,z)\in R_{K}|\prod_{m=1}^{n}Multinumber_{m}(x/_{\sim_{K,z}},x)=d\rbrace
$$
is an  $\mathcal{L}$-definable subset of $R_{K}$.
Consider a rational motivic constructible function $g\in\tilde\cC(R)$ such that  
$$
g_K(x,z)=\sum_{d=1}^{Q}\dfrac{\11_{R_{K}(d)}(x,z)}{d}.
$$
Here, the denominator $d$ can be viewed as a set of $d$ $\emptyset$-definable points in the residue field sort. 
Next we define the rational motivic constructible function $\Phi\in \tilde{\mathcal{C}}_{+}(R)$ by
\begin{center}
$\Phi= \dfrac{g\cdot \LL^{f}}{[D]}$,
\end{center}
where the map $D\to R$ comes from the coordinate projection.
By the definition, $\Phi$ is $Z$-integrable is we have that $f_{L,z}$ is bounded above on $R_{L,z}$ for each $z\in Z(L)$ and each $\cT$-field $L$. The fact that $f_{L,z}$ is bounded above is a definable condition, hence, up to replacing $f$ so that it is zero if $f_{L,z}$ is not bounded above, we may suppose that $f_{L,z}$ is bounded above on $R_{L,z}$ for each $z\in Z(L)$ and each $\cT$-field $L$ and thus that $\Phi$ is $Z$-integrable. 
Set 
$$
F=\mu_{/Z}(\Phi) \mbox{ in } \tilde{\mathcal{C}}_{+}(Z).
$$
Finally we prove that $a_{K,z}=F_{K}(z)$ for all $z\in Z(K)$ and all local fields $K$ with $\charac(k_{K})>M$.
If $\charac(k_{K})>M$ for well-chosen $M$ we have  
$$
F_{K}(z) = \mu_{/\{z\}}(\Phi|_{R_{K,z}})=\int_{x\in R_{K,z}}\frac{g_{K}(x,z)q_{K}^{f_{K}(x,z)}}{\#(D_{K,x,z})}|dx|
$$
where $|dx|$ is the Haar measure on $K^{n}$.
Let $J_{z}$ be a set of $a_{K,z}$ representatives of $\sim_{K,z}$.
Thus, 
\begin{align*}
&\int_{x\in R_{K,z}}\frac{g_{K}(x,z)q_{K}^{f_{K}(x,z)}}{\#(D_{K,x,z})}|dx|\\
= &\sum_{x\in J_{z}}\int_{y\in x/_{\sim_{K,z}}}\frac{g_{K}(y,z)q_{K}^{f_{K}(y,z)}}{\#(D_{K,y,z})}|dy|
\end{align*}
Let us for a moment fix $x$ and write $X=x/_{\sim_{K,z}}$ and  $d_{m}(y)=Multinumber_{m}(y/_{\sim_{K,z}},y)$. Then, using Fubini's theorem, we have
\begin{align*}
&\int_{y\in X}\frac{g_{K}(y,z)q_{K}^{f_{K}(y,z)}}{\#(D_{K,y,z})}|dy|\\
=& \int_{\overline{y}\in X(n-1)}\int_{y_{n}\in X_{\overline{y}}}\frac{q_{K}^{f_{K,z,1}(x)+...+f_{K,z,n}(x)}}{\prod_{m=1}^{n}[d_{m}(\overline{y},y_{n})\times\#(D_{K,z,(\overline{y},y_{n}),m})]}
|dy_{n}|\cdot  |d\overline{y}|\\
=&\int_{\overline{y}\in X(n-1)}\frac{q_{K}^{f_{K,z,1}(x)+...+f_{K,z,n-1}(x)}}{\prod_{m=1}^{n-1}[d_{m}(\overline{y})\times\#(D_{K,z,(\overline{y},y_{n}),m})]}
|d\overline{y}|,
\end{align*}
where $f_{K,z,i}(y)=f_{K,z,i}(x)$ for all $y\in X$ and $1\leq i\leq n$ and where $d_{m}(\overline{y}):=d_{m}(\overline{y},y_{n})$ does not depend on $y_{n}$ when $y=(\overline{y},y_{n})$ varies in $X$ for all $1\leq m\leq n-1$.
The last equality comes from:
\begin{align*}
&\int_{y_{n}\in X_{\overline{y}}}\frac{q_{K}^{f_{K,z,n}(x)}}{d_{n}(\overline{y})\times\#(D_{K,z,(\overline{y},y_{n}),n})}
|dy_{n}|\\
=&\frac{q_{K}^{f_{K,z,n}(x)}}{d_{n}(\overline{y})}\sum_{i=1}^{d_{n}(\overline{y})}\int_{y_{n}\in B_{i}\cap X_{\overline{y}}}\frac{1}{\#(D_{K,z,(\overline{y},y_{n}),n})}|dy_{n}|\\
=&\frac{q_{K}^{f_{K,z,n}(x)}}{d_{n}(\overline{y})}\sum_{i=1}^{d_{n}(\overline{y})}\Vol(B_{i}\cap X_{\overline{y}})\frac{1}{N_{x/_{\sim_{K,z}},\overline{y},B_{i}}}\\
=&\frac{q_{K}^{f_{K,z,n}(x)}}{d_{n}(\overline{y})}\sum_{i=1}^{d_{n}(\overline{y})}q_{K}^{-f_{K,z,n}(x)}\\
=&1
\end{align*}
where $B_{1},...,B_{d_{n}(\overline{y})}$ are the balls with volume $q_{K}^{-f_{K,z,n}(x)+1}$ which have nonempty intersection with $X_{\overline{y}}$; indeed, one sees that $B_{i}\cap X_{\overline{y}}$ is the union of $N_{x/_{\sim_{K,z}},\overline{y},B_{i}}$ many disjoint balls of volume $q_{K}^{-f_{K,z,n}(x)}$. 
By applying Fubini's theorem with a similar calculation to each of the remaining $n-1$ variables we deduce that 
$$
\int_{y\in X}\frac{g_{K}(y,z)q_{K}^{f_{K}(y,z)}}{\#(D_{K,y,z})}|dy|=1.
$$
We conclude that 
$$F_K(z) = \mu_{/\{z\}}(\Phi|_{R_{K,z}})= \#J_{z}=a_{K,z}$$
for all local fields $K$ with $char(k_{K})>M$ and $z\in Z(K)$. The main theorem is proved.
\end{proof}

\subsection*{Acknowledgement}

I would like to thank Raf Cluckers for his helpful discussions and suggestions during the preparation of this paper. I also thank to Pablo Cubides Kovacsics for some comments. The author are supported by the European Research Council under the European Community's Seventh Framework Programme (FP7/2007-2013) with ERC Grant Agreement nr. 615722
MOTMELSUM, and thank the Labex CEMPI  (ANR-11-LABX-0007-01).

\end{document}